\documentclass[11pt]{article}
\input{amssym.def}
\input{amssym}
\usepackage[active]{srcltx}
\usepackage{color}
\usepackage{tikz}
\usepackage{relsize}

\usepackage{hyperref}
\hypersetup{
    colorlinks=true, 
    linktoc=all,     
    linkcolor=blue,  
}

\newenvironment{annotacia}{\centerline{\sc Abstract}\vspace{2mm}\narrower\narrower\sf}

\def\theequation{\thesection.\arabic{equation}}
\makeatletter\@addtoreset{equation}{section}\makeatother

\def\nn{\nonumber}\def\lb{\label}
\def\be{\begin{equation}}\def\ee{\end{equation}}
\def\ba{\begin{eqnarray}}\def\ea{\end{eqnarray}}

\newcounter{theorem}\makeatletter
\@addtoreset{theorem}{section}\makeatother

\newtheorem{prop}[theorem]{Proposition}
\newtheorem{rem}[theorem]{Remark}

\newtheorem{def-lem}[theorem]{Definition-Lemma}
\newtheorem{def-prop}[theorem]{Definition-Proposition}

\begin{document}
\title{ }
\begin{center}
{\Large \textbf{Iterative construction of the R-matrices in arbitrary dimensions}}
	
\vspace{1cm} {\large \textbf{Petr Pivovarov$^{\diamond}$ 
and Pavel Pyatov$^{\diamond\,\dag}$}}

\vskip .3cm $^{\diamond}${National Research University "Higher School 
of Economics", 20 Myasnitskaya street, Moscow
101000, Russia}
	
\vskip .3cm $^{\dag}${Bogoliubov Laboratory of Theoretical Physics, 
JINR, 141980 Dubna, Moscow region, Russia}
\end{center}

\begin{annotacia}\noindent 
We investigate a special ansats that allows for an iterative solution of the constant Yang-Baxter equation.
Testing this Ansatz, we construct four sequences of the constant R-matrices. 
In each sequence the R-matrices act on the tensor squares of vector spaces of linearly growing dimensions. Each R-matrix also depends on a single complex parameter.

By analyzing the spectra of the R-matrices, we conclude that the first and the third series are associated with the symmetric tensor representations of the quantum groups $U_q(\frak{sl}(2))$ and $U_q(\frak{sl}(2|1))$, respectively. The other two series appear to be  related to representations of  quantum groups in the case  where $q$ is a root of unity. The elements of the second series are associated with nilpotent representations of $U_q(\frak{sl}(2))$ (representations $T_{00\lambda}$ in the notation of \cite{KSch}).

We also check the correspondence  between these R-matrices and link invariants. Applying the Reshetikhin–Turaev procedure to the first, second, and fourth sequences yields, respectively, colored Jones polynomials, ADO invariants, and Alexander polynomials.
To our surprise, a similar calculation for the third series yields a trivial result, identically equal to 1.

\end{annotacia}

\newpage
\tableofcontents
\bigskip\bigskip\bigskip
\section{Introduction}\lb{sec1}

Over the past half-century, the matrix braid relation -- also known as the constant Yang-Baxter equation --  has become one of the fundamental equations of modern mathematics and mathematical physics. 
A solution to this equation is an operator, commonly referred to as an R-matrix, acting on the tensor square of a finite- or infinite-dimensional space $V$.
R-matrices have found numerous applications in the theory of quantum groups and the construction of link invariants, as well as in the theory of integrable systems in statistical physics, quantum mechanics, and models of stochastic processes.

The problem of classifying R-matrices appears unrealistic, except for the simplest case where they act on a space $V$ of dimension 2 \cite{H1,H2}.
Even in dimension 3, R-matrix solutions to the braid relation have been investigated only within the framework of specific ansätze. The upper-triangular R-matrices were classified in \cite{H3}.	In \cite{EO}, all R-matrices associated with quantum groups of type $GL(3)$ were found.  Finally, a family of R-matrices satisfying the "additive charge conservation" condition was constructed in \cite{HMR}.

Another line of research examined various ansätze enabling the construction of families of R-matrices acting on spaces $V$ of arbitrary dimensions.  
The starting point for the research here was the so-called ice condition, which requires the R-matrix to have non-zero components only at the same positions as the identity operator or the permutation operator. The family of ice-type R-matrices
was classified in \cite{O}. It turns out that these R-matrices have at most two distinct eigenvalues  and, consequently, they generate representations of Iwahori–Hecke algebras (see, e.g., \cite{Jo}). These R-matrices are associated with the quantum linear groups 
and supergroups,
 see \cite{O,I,GPS}.

Two generalizations of the ice conditions were considered.
The rime condition examined in \cite{OP} encompasses a broader family of R-matrices possessing at most two eigenvalues. 
Another generalization \cite{I} leads to the construction of a family of orthogonal, symplectic, and orthosymplectic 
R-matrices. These R-matrices have at most three distinct eigenvalues
and generate representations of Birman–Murakami–Wenzl algebras \cite{BW,M1}.

In the present work, we consider yet another generalization of the ice Ansatz. We combine the conditions of  additive charge conservation \cite{HMR} and triangularity \cite{H3} (respectively, degree conservation and  upper anti-triangularity in our notation, see the next section).
This combination allows us to seek solutions to the braid relation iteratively, using induction on the dimension of the space $V$.  

Without aiming to directly classify all possible solutions, we obtain, within the framework of this Ansatz, four families of R-matrices that act on spaces $V$ of arbitrarily high dimension. Each R-matrix contains  a single essential parameter explicitly related to its eigenvalues.\footnote{We fix the values of the non-essential gauge parameters and twist parameters. The twist parameters are reconstructed in the appendix.} Such a parameterization allows us to control the spectrum of R-matrices: the number of their distinct eigenvalues grows linearly with the dimension of the space $V$.  Two of these families -- one corresponding to   case $a_1)$ in proposition \ref{prop2}  and the other in proposition \ref{prop3} -- are evidently related to evaluations of universal R-matrices of the quantum groups $U_q({\frak sl}(2))$ and $U_q({\frak sl}(2|1))$ in their symmetric tensor representations,  that is, in the irreducible representations corresponding to single-row Young diagrams. The other two --  case $a_2)$ in proposition \ref{prop2} and another family in proposition \ref{prop4} -- are related to representations of the quantum groups at  roots of unity. In the main text, we discuss in somewhat greater detail the connection between the resulting R-matrices and the quantum groups, as well as  their relation to link invariants. 
The diagonal twist transformation formulas for all four families of R-matrices are collected in the appendix.

In concluding this introductory section, we emphasize that the above overview of research on R-matrix solutions to the braid relation is by no means exhaustive.
We have mentioned only those works that are directly relevant to our current study.
The list of other important and interesting lines of research includes but is not limited to: the quantization of classical r-matrices \cite{GG,GH,IO}; the application of  Nichols algebras for the construction of R-matrices \cite{GK,KM};
construction of R-matrices in terms of representations of Temperley–Lieb algebras, Hecke algebras, and partition algebras \cite{AFFKRR,B,MSPK}; 
as well as the use of various Ansätze, including the unitary \cite{Le,GR}, stochastic \cite{PPTWZ}, and monomial (related to set-theoretic solutions of the Yang–Baxter equation \cite{ESS}) Ansätze.

\section{Preliminaries}\lb{sec2}

\subsection{Definitions and notation}\lb{subsec2.1}

Let $V$ be a finite dimensional $\Bbb C$-linear space. We denote $\mbox{dim}\, V=N$ and fix
a basis $\{e_i\}_{i=1}^{N}$ in $V$. We also fix the induced bases in tensor power spaces $V^{\otimes n}$:
$\{e_{i_1}\otimes e_{i_2}\otimes \ldots \otimes e_{i_n}\}_{i_k=1}^{N}$,
assuming lexicographical ordering of the basis vectors. Thus, any operator $A\in {\rm End}(V^{\otimes n})$ is identified with its matrix 
$A_{i_1,\ldots ,i_n}^{\,j_1,\ldots ,j_n}$ in a fixed basis.

Assigning to each basis vector $e_i$ a {\em degree function} ${\rm deg} (e_i):= i$ we extend it to   basis vectors in tensor product spaces $V^{\otimes n}$ additively:\,
$
{\rm deg}(e_{i_1} \otimes e_{i_2} \otimes... \otimes e_{i_n}):=\sum_{k=1}^{n} i_k
$.\,
That gives a decomposition of the space $V^{\otimes n}$ into a direct sum of its degree subspaces
\be
\lb{deg-spaces}
\textstyle
V^{(n;m)}:= {\rm Span}\{e_{i_1} \otimes...\otimes e_{i_n}|\,{\textstyle\sum_{k=1}^n} i_k = m\}, \quad
V^{\otimes n}= \bigoplus_{m=n}^{ n N }V^{(n;m)}.
\ee

An operator $A\in {\rm End}(V^{\otimes n})$ acting invariantly in all subspaces $V^{(n;m)}$ is called {\em degree-conserving}.
The matrix of the degree-conserving operator has a block-diagonal form, we denote $A^{(m)}$ its block acting in the subspace $V^{(n;m)}$. 

We call an operator $A$ {\em upper anti-triangular} if all elements of its matrix lying strictly below the main anti-diagonal are zero. 
\smallskip

In this paper we consider exclusively operators acting in the spaces $V$, $V^{\otimes 2 }$ and $V^{\otimes 3 }$. 

Let us denote ${\rm Id}_V$ the identity operator in space $V$. With any operator $A\in {\rm End}(V^{\otimes 2})$ we associate a pair of operators 
$A_{1,2}\in {\rm End}(V^{\otimes 2})$:
\be
\lb{A1,2}
A_1:= A\otimes{\rm Id}_V, \quad A_2:= {\rm Id}_V\otimes A.
\ee
We use the capital letter $R$ to denote an invertible operator $R\in {\rm Aut}(V^{\otimes 2})$ satisfying the equation
\be
\lb{YB}
R_1 R_2 R_1 - R_2 R_1 R_2 =0. 
\ee
Traditionally, this equation is called {\em braid relation}, or {\em constant Yang-Baxter equation};  its solution $R$ is called {\em $R$-matrix}. In the sequel, we  use the shorthand notation 
\be
\lb{Y}
Y:=R_1 R_2 R_1 - R_2 R_1 R_2
\ee
for the operator appearing on the left-hand side of the braid relation.

\subsection{R-matrix: the Ansatz}\lb{subsec2.2}

In the main text we investigate solutions of the constant Yang-Baxter equation satisfying additional conditions
\be
\lb{ansatz}
\begin{array}{ll}
\bullet&\mbox{$R$ is the degree-conserving operator;}\\
\bullet&\mbox{all blocks $R^{(m)}$, $m=2,\dots 2N$, are upper anti-triangular operators;}\\
\bullet										&\mbox{$R$ is normalized so that $R^{(2)}=1$;}\\
\bullet&\mbox{blocks of $R$ do not contain $2\times 2$ anti-diagonal sub-blocks.} 
\end{array}
\ee
\begin{rem}{\small\rm
		Note that the lower anti-triangularity condition can be equivalently used in the Ansatz. A one-to-one correspondence between families of upper and lower antitriangular R-matrices can be established by the inversion map $R\mapsto R^{-1}$, or by the permutation map $R\mapsto P R P$, where 
		$P(e_i\otimes e_j)=e_j\otimes e_i$ is the permutation operator.}
\end{rem}
\begin{rem}{\small\rm
The normalization condition is not a restriction but rather a suitable choice of the representative in an equivalence class of the R-matrices: we choose R-matrix whose first eigenvalue equals 1.	The last condition is introduced in the Ansatz to exclude from consideration the large number of permutation-type R-matrices (see, e.g. considerations in the next section).	}
\end{rem}

Let us clarify our motivation in this Ansatz.

Degree-conservation condition or, in terminology of Ref. \cite{HMR}, {\em additive charge conservation}  implies a block-diagonal structure for the operator $Y\in {\rm End}(V^{\otimes 3})$  (\ref{Y}). In this case, solving the braid relation (\ref{YB}) reduces to considering the vanishing conditions for the smaller blocks $Y^{(m)}$, $m=3, \ldots, 3N$. Moreover\vspace{-1mm}
\begin{itemize}
	\item[-\,] 
	the R-matrix block $R^{(m)}$ appears for the first time in the braid relation block
	$Y^{(m+1)}$, and due to the anti-triangularity of $R^{(m)}$,  components of $Y^{(m+1)}$ can only be quadratic-linear in the components of $R^{(m)}$; 
	\item[-\,] 
	braid relations in block $Y^{(m)}$ are identical for degree-conserving R-matrices acting in any space $V$ of dimension $N\geq m-2$. 
\end{itemize} 
\vspace{-1mm}
These two properties open a possibility of an iterative solution of the braid relation.

Note additionally that the upper anti-triangularity condition for $R$ significantly simplifies the system of  cubic equations (\ref{YB}). Some of these equations admit complete factorization, and under certain additional assumptions,\footnote{See, e.g., assumption (\ref{cond}) below.} the system of equations from the block $Y^{(m+1)}$ can be solved for the components of the matrix $R^{(m)}$ in a general form.
\medskip

Let us also comment on the equivalence transformations and the parameterization of solutions of the braid relation.

First of all, there are trivial {\em gauge transformations} on the set of solutions of the braid relation. They are generated by arbitrary changes of the basis of space $V$. In general, only diagonal gauge transformations, i.e., rescalings of the basis vectors are compatible with the Ansatz. We will always single out a unique representative in each diagonal gauge family of solutions to the braid relation.

There is also a nontrivial family of the {\em twist transformations} of the R-matrices. A general description of this family is missing, but in many cases including the case of our Ansatz  a simple subfamily of diagonal twist transformations can be explicitly described (see \cite{Resh,I,IOP3}). These are some of the diagonal transformations of the basis in $V^{\otimes 2}$ that do not reduce to rescaling a basis in $V$ (for convenient criteria to identify simple twist transformations, including diagonal ones, see \cite{PMSK}, Section 2.2).
The diagonal twist transformations of the R-matrices found in sections \ref{sec4},~\ref{sec5}  are described in the Appendix.

In the main text we use a parameterization of solutions of the Yang-Baxter equation (\ref{YB}), which is directly related to the eigenvalues of the R-matrices. 
Such a parameterization significantly simplifies the search for solutions to the braid relation, at least within the framework of our Ansatz. 
We believe that the convenience of parametrization through eigenvalues of R-matrices is directly related  to the fact that irreducible representations of the braid group $B_3$ in low-dimensional spaces are classified primarily (up to discrete parameter) by the eigenvalues of
 the Artinian braid generators (see \cite{TW,PT}).

\section{Simple examples}\lb{sec3}

In this section, we classify -- within the framework of the Ansatz (\ref{ansatz}) -- the R-matrices acting on the tensor squares of spaces $V$ of dimensions 2 and 3. Naturally, in these cases we simply reproduce those R-matrices from the classification lists in  \cite{H2,HMR} that  are composed of anti-triangular blocks. The aim of this section is to provide a base for a subsequent induction on the dimension of space $V$ and to demonstrate the key stages of the iterative block-by-block solution of the braid relation.

\subsection{The case $\mbox{dim}\, V = 2$}
In this case the R-matrix $R\in {\rm Aut}(V^{\otimes 2})$ satisfying the Ansatz has 3 blocks.
Let us introduce notation for their components
$$
R^{(2)} =(1), \qquad R^{(3)} =\left(
\begin{array}{cc}
	\alpha_{11} &	\alpha_{12}
	\\
 \alpha_{21} & 	 0
\end{array}\right), \qquad
R^{(4)} =(\beta).
$$

Blocks of the operators $R_1, R_2\in {\rm Aut}(V^{\otimes 3})$ then has a form
\ba
\nn
R_1^{(3)} =(1), \qquad R_1^{(4)} =\left(
\begin{array}{ccc}
 1 & 0 & 0
	\\
 0&	 \alpha_{11} &  \alpha_{12}
	\\
 0&	 \alpha_{21} &  0
\end{array}\right), \qquad
R_1^{(5)} =\left(
\begin{array}{ccc}
	\alpha_{11} & \alpha_{12}& 0
	\\
	\alpha_{21} & 0 & 0
	\\
	0 & 0& \beta
\end{array}\right), \qquad
R_1^{(6)} =(\beta);
\\[10pt]
\nn
R_2^{(3)} =(1), \qquad R_2^{(4)} =\left(
\begin{array}{ccc}
	\alpha_{11} & \alpha_{12}& 0
	\\
	\alpha_{21} & 0 & 0
	\\
	0 & 0& 1
\end{array}\right), \qquad
R_2^{(5)} =\left(
\begin{array}{ccc}
\beta & 0 & 0
\\
0&	\alpha_{11} & \alpha_{12}
\\
0&	\alpha_{21} & 0
\end{array}\right), \qquad
R_2^{(6)} =(\beta).
\ea
Using these blocks, we construct braid relation blocks $Y^{(3)},\dots ,Y^{(6)}$, and obtain the following equations.

The first and last blocks of the braid relation -- in our case, $Y^{(3)}$ and $Y^{(6)}$ --  always vanish identically.

The braid relations in block $Y^{(4)}$ yield a single non-trivial condition in the component $Y^{(4)}_{\,1,1}$:
\be
\lb{c-1}
\alpha_{11}(\alpha_{11}-1+\alpha_{12}\alpha_{21}) =0.
\ee
By virtue of the last condition of the Ansatz $\alpha_{11}\neq 0$; due to the invertibility of the R-matrix $\alpha_{12}\alpha_{21}\neq 0$.
Thus, we can choose the following suitable parameterization for the bock $R^{(3)}$ that satisfies condition (\ref{c-1})
\be
\lb{R(3)}
R^{(3)} =\left(
\begin{array}{cc}
	1-v &	x
	\\
	v/x & 	 0
\end{array}\right).
\ee
Here $-v$ is the second eigenvalue of the block $R^{(3)}$ (its first eigenvalue is necessarily equal to 1); $x$ is the diagonal twist parameter of the R-matrix. 

The only non-trivial condition in the block $Y^{(5)}$ arises from the component $Y^{(5)}_{\,1,1}$:
\be
\lb{c-2}
\beta(\beta -1)(\beta +v) =0.
\ee
The invertibility of the R-matrix requires that $\beta\neq 0$. Thus, we have two possibilities: $\beta=1$, or $\beta=-v$. The two corresponding R-matrices for the two-dimensional space $V$ that satisfy the Ansatz are of the form:
\be
\lb{dim2-R}
R_{DJ}=\left(
\begin{array}{lr|lr}
	1 &  &  &  
	\\
	 & 1-v & x &
	 \\
	  \hline
	 & v/x &   &
	 \\
	 &  &  &  1
\end{array}
\right),\qquad
R_{KS}=\left(
\begin{array}{lr|lr}
	1 &  &  &  
	\\
	& 1-v & x &
	\\
	\hline
	& v/x &   &
	\\
	&  &  &  -v
\end{array}
\right).
\ee
Here, we have left blank the positions occupied by the  zero components of the R-matrices.

\begin{rem}
	{\rm 
These R-matrices are simplest examples of, respectively,  {\em Drinfeld-Jimbo's} \cite{D,Ji} and {\em Kulish-Sklyanin's} \cite{KS} families of the R-matrices. Both possess quadratic minimal polynomial and, consequently, give rise to representations of the Iwahori-Hecke algebras (see,e.g., \cite{Jo}).  The R-matrices  $R_{DJ}$ and $R_{KS}$ can be assigned the types $GL(2)$ and $GL(1|1)$, respectively (see \cite{GPS}), as they are used in the construction of the quantum (super-)spaces and the quantum (super-)groups of the corresponding types.
}
\end{rem}

\subsection{The case $\mbox{dim}\, V = 3$}

In this case, the R-matrix $R\in {\rm Aut}(V^{\otimes 2})$ satisfying the Ansatz consists of five blocks.
The first two blocks coincide with the blocks from the previous case: $R^{(2)}=1$, and $R^{(3)}$ is given by (\ref{R(3)}). 
For the components of the remaining blocks, we introduce the following notation
$$
R^{(4)} =\left(
\begin{array}{ccc}
a_{11} &	a_{12} & a_{13}
	\\
	a_{21} & 	a_{22} & 0
		\\
	a_{31} & 0& 0
\end{array}\right), \qquad
 R^{(5)} =\left(
\begin{array}{cc}
	b_{11} &	b_{12}
	\\
	b_{21} & 	 0
	\end{array}\right), \qquad
R^{(6)} =(c).
$$
The braid relations in the block $Y^{(4)}$ are satisfied thanks to the choice of $R^{(3)}$.  

Conditions on the components of the $R$-matrix block $R^{(4)}$
arise from the braid relations in the block $Y^{(5)}$, which is constructed from the following blocks of operators $R_1$ and $R_2$
$$
R_1^{(5)} =\left(
\begin{array}{cccccc}
	1 &	0 & 0 & 0 & 0&0
	\\
	0& 	\alpha_{11} & 0 & \alpha_{12}& 0&0
	\\
	0 & 0 & a_{11}&0 & a_{12} & a_{13}
		\\
	0& 	\alpha_{21} & 0 & 0& 0&0
		\\
	0 & 0 & a_{11}&0 & a_{22} & 0
		\\
	0 & 0 & a_{31}&0 & 0 & 0
\end{array}\right), \qquad
R_2^{(5)} =\left(
\begin{array}{cccccc}
	a_{11}&	a_{12} & a_{13} & 0 & 0&0
	\\
	a_{21}& a_{22} & 0 & 0& 0&0
	\\
	a_{31}& 0 & 0&0 &0 & 0
	\\
	0& 	0 & 0 & \alpha_{11}& \alpha_{12}&0
	\\
	0 & 0 & 0&\alpha_{21} &0 & 0
	\\
	0 & 0 & 0&0 & 0 & 1
\end{array}\right).
$$
Relations in the components $Y^{(5)}_{12}$, $Y^{(5)}_{14}$, $Y^{(5)}_{41}$  are quadratic with respect to the components of $R^{(4)}$.\footnote{Components $Y^{(5)}_{21}$, $Y^{(5)}_{41}$, $Y^{(5)}_{15}$, $Y^{(5)}_{23}$, $Y^{(5)}_{32}$ give the same relations.} They give rise to two branches of solutions:
\ba
\lb{b-a}
&
\mbox{\em either \,
	} a_{12}=a_{21}=0, &
\\
\lb{b-b} 
&
\mbox{\em or,\,
 if\,
$a_{12}$, or $a_{21}\neq 0$, then}\;\; a_{13}=x^2,\;\; a_{31}=v^2/x^2,\;\; a_{11}=(1-v)(1-a_{22}). &
\ea
Non-trivial relations also appear in the components $Y^{(5)}_{11}$ and $Y^{(5)}_{22}$, respectively,
\ba
\lb{b-c}
a_{11}(a_{11}-1+a_{13}\, a_{31})+a_{12}\,a_{21}&=&0,
\\
\lb{b-d} 
(1-v)(a_{22}-1)(a_{22}+v) +a_{12}\,a_{21}&=&0.
\ea

If the conditions (\ref{b-a}) are satisfied,
then, given that the $2\times 2$ sub-block in the matrix $R^{(4)}$ is not anti-diagonal,
we have $a_{11}\neq 0$. By the same reason $v\neq 1$ in the block $R^{(3)}$ (\ref{R(3)}).
In this case, relations (\ref{b-c}) and (\ref{b-d}) reduce to the form of eqs. (\ref{c-1}) and (\ref{c-2}) considered earlier, and their solution has the form
\be
\lb{R4-1}
R^{(4)} =\left(
\begin{array}{ccc}
1-v' &	0 & x'
	\\
	0 & 	a_{22} & 0
	\\
	v'/x' & 0& 0
\end{array}\right),\quad \mbox{where $a_{22}\in \{1,-v\}$}.
\ee
Here, $-v'$ and $x'$ are the new spectral value and the new diagonal twist parameter of the R-matrix, respectively.\smallskip

Alternatively, if the conditions (\ref{b-b}) are fulfilled, the condition (\ref{b-c}) becomes a consequence of (\ref{b-b}) and (\ref{b-d}). By choosing the appropriate parameterization
$$
a_{22} = q v, \quad a_{12}= g x (1-q v),
$$
where $q$ and $g$ are new parameters, it is easy to find a solution to equations (\ref{b-b}), (\ref{b-d})
\be
\lb{R4-2}
R^{(4)} =\left(
\begin{array}{ccc}
	(1-v)(1-qv) &	g x(1-q v) & x^2
	\\
	(1-v)(1+q)/(g x) & 	q v & 0
	\\
	v^2/x^2 & 0& 0
\end{array}\right).
\ee
Here, the parameters $q$ and $v$ determine the spectrum of the block $R^{(4)}$: ${\rm Spec}\, R =\{1,-v,qv^2\}$; $x$ is the diagonal twist parameter; and $g$ is an inessential gauge parameter describing the relative normalization of the basis vectors in space $V$.
\smallskip

Turning to the braid relations in blocks $Y^{(6)}, \dots ,Y^{(8)}$, for the sake of brevity, we shall no longer write out the explicit forms of the corresponding blocks of the operators $R_1$ and $R_2$. Their structure is evident from the examples already considered.

For the R-matrix block $R^{(4)}$ (\ref{R4-1}), the only essentially new condition arises in the braid relation block $Y^{(6)}$: the vanishing of its component $Y^{(6)}_{\,33}=b_{11}(v-v')$ imposes restriction\,\footnote{Here, we again use the last condition of the Ansatz, which requires that $b_{11} \neq 0$.} 
\be 
\lb{v-v}
v'=v
\ee
on the spectral values of the R-matrix. 
All other relations for the components of blocks $R^{(5)}$ and $R^{(6)}$ take the form of eqs. (\ref{c-1}) and (\ref{c-2}). The final solution is as follows:
\be 
\lb{R56-a}
R^{(5)} =\left(
\begin{array}{cc}
	1-v &	x''
	\\
	v/x'' & 	 0
\end{array}\right), \qquad
R^{(6)} \in \{1,-v\},
\ee
where $x''$ is yet another diagonal twist parameter.\smallskip

For the R-matrix block $R^{(4)}$ (\ref{R4-2}), several new relations appear in block $Y^{(6)}$. Namely, we have
\ba
\lb{new-1}
Y^{(6)}_{24}, Y^{(6)}_{43} &\Rightarrow &
\left\{
\begin{array}{l}
 (1-q v) (x b_{21} - q^2 v^2 ) =0,
\\
(1+q)(1-v)(x b_{21} - q^2 v^2 ) =0;
\end{array}	
\right.
\\
\lb{new-2}
Y^{(6)}_{34}, Y^{(6)}_{42} &\Rightarrow &
	\left\{
	\begin{array}{l}
		(1-q v) ( b_{12} - x q^2 v ) =0,
		\\
		(1+q)(1-v)( b_{12} - x q^2 v ) =0;
	\end{array}	
	\right.
	\\
	\lb{new-3}
	Y^{(6)}_{33} &\Rightarrow & \quad\,(1-v)(b_{11}- q v (1+q)(1-q v)) =0;
	\\
	\lb{new-4}
	Y^{(6)}_{22} &\Rightarrow & \quad\,(1-v)\left(b_{12}b_{21}-  v^2 \left(1+q(1+q)(1-q v)\right)\right) =0.
\ea
On the left-hand sides of these relations, we have indicated the components of matrix $Y^{(6)}$ to which the relations given on the right-hand sides correspond.

Once again, factor $(1-v)$ is non-zero; otherwise, the matrix $R^{(3)}$ would be anti-diagonal. For the same reason, we cannot simultaneously set $(1+q)=0$ and $(1-q v)=0$ in the matrix $R^{(4)}$ (\ref{R4-2}). Hence,  the relations (\ref{new-1})-(\ref{new-4}) imply the vanishing of the rightmost factors on their left-hand sides. Solving eqs.(\ref{new-1})-(\ref{new-3}) we find
\be 
\lb{R5-b}
R^{(5)} =\left(
\begin{array}{cc}
	q v (1+q)(1-q v) &	x q^2 v
	\\
	q^2 v^2/x & 	 0
\end{array}\right),
\ee
and then, eq.(\ref{new-4}) constraints the parameters $q$ and $v$:
\be 
\lb{u-v}
(1+q+q^2)(1-q^2 v)=0.
\ee

Passing to the braid relations in blocks $Y^{(7)}$ and $Y^{(8)}$, we do not obtain new restrictions on the R-matrix parameters but merely fix the last of its blocks: 
\be 
\lb{R6-b}
R^{(6)}= q^4 v^2.
\ee

To summarize, below we collect all the R-matrices satisfying the Ansatz and acting on spaces $V$ of dimension 3:
\ba
\lb{dim3-R}
\phantom{a}\hspace{-50mm}
R_{A}=\left(
\begin{array}{lcr|lcr|lcr}
\scriptstyle	1 &  &  &  & & & & &
	\\
	 &\scriptstyle 1-v &  & \scriptstyle 1 & & & & &
	\\
	 &  &\scriptstyle 1-v  &  & & &\scriptstyle 1 & &
	\\
	\hline
	&\scriptstyle v &  & & & & & &
	\\
	&  &  & &\scriptstyle \lambda_1 & & & &
	\\	
	&  &  & & &\scriptstyle 1-v &  &\scriptstyle 1 &
	\\
	\hline
		&  &\scriptstyle v & & & & & &
	\\
		&  &  & & &\scriptstyle v & & &
	\\
		& &  & & & & & &\scriptstyle \lambda_2
\end{array}
\right)\! ,\qquad\mbox{where}\quad \lambda_{1,2}\in\{1,-v\};
\\[8pt]
R_{B}=\left(
\begin{array}{lcr|lcr|lcr}
	\scriptstyle	1 &  &  &  & & & & &
	\\
	&\hspace{-2mm}\scriptstyle 1-v &  & \scriptstyle 1 & & & & &
	\\
	&  &\hspace{-4mm}\scriptstyle (1-v)(1-q v)  &  &\hspace{-2mm}\scriptstyle 1-q v & &\scriptstyle 1 & &
	\\
	\hline
	&\hspace{-2mm}\scriptstyle v &  & & & & & &
	\\
	&  & \scriptstyle (1- v)(1+q) & &\hspace{-2mm}\scriptstyle q v & & & &
	\\	
	&  &  & & &\hspace{-4mm}\scriptstyle q v (1- q v)(1+q) & & \scriptstyle q^2 v &
	\\
	\hline
	&  &\scriptstyle v^2 & & & & & &
	\\
	&  &  & & &\scriptstyle q^2v^2 & & &
	\\
	& &  & & & & & &\scriptstyle q^4 v^2
\end{array}
\right)\!, \; \mbox{where}\; {\scriptstyle (1-q^2 v)(1+q+q^2)=0}.
\ea
We do not show zero components in these R-matrices. We have also set all gauge and twist parameters of the R-matrices to unity; recovering them would be straightforward. 

The R-matrix $R_A$ correspond to the first branch  (\ref{b-a}) of the solution to the braid relation. It is given by formulas (\ref{R(3)}), (\ref{R4-1}), (\ref{v-v}), and (\ref{R56-a}). The R-matrix $R_B$ represent the second branch
(\ref{b-b}), given by formulas (\ref{R(3)}), (\ref{R4-2}), and (\ref{R5-b})-(\ref{R6-b}).

\begin{rem}
	{\rm   For different choices of the parameters $\lambda_1$ and $\lambda_2$, $R_A$ yields either the Drinfeld–Jimbo $R$-matrix of type $GL(3)$   ($\lambda_1=\lambda_2=1$) or the Kulish-Sklyanin $R$-matrices of types $GL(2|1)$ ($\lambda_1=1, \lambda_2=-v$ or $\lambda_1=-v, \lambda_2=1$) and $GL(1|2)$ ($\lambda_1=\lambda_2=-v$). All these $R$-matrices give rise to representations of the Iwahori–Hecke algebras.
		
		The R-matrix $R_B$ can be obtained from the universal R-matrix of the quantum group $U_q(\frak{sl}(2))$
		evaluated in the tensor squares of the  spin-1 representation (the case $v=1/q^2$) or  the nilpotent representation at $q=\sqrt[3]{1}$ (the case $1+q+q^2=0$). For  details and  references, see remark
		\ref{rem5.2} below.		
	}
\end{rem}

\section{Construction of the pre-equatorial blocks}\lb{sec4}
We call the largest block $R^{(N+1)}$ ($N=\mbox{dim}\, V$)  of a degree-preserving R-matrix the {\em equatorial} block.
The preceding blocks $R^{(2)},\dots ,R^{(N)}$ and   the subsequent blocks  $R^{(N+2)},\ldots ,R^{(2N)}$ are called, respectively, {\em pre-equatorial} and {\em post-equatorial}.\smallskip

Let us introduce a set of $i\times i$ matrices, $J^{[i]}$, $i\in {\Bbb N}$, depending on two complex  parameters $q, v\in{\Bbb C}\setminus \{0\}$:
\be
\lb{JJ}
J^{[i]}_{ab}(q,v):=q^{(a-1)(b-1)} v^{(a-1)}\,
\Bigl[ { i-b\atop a-1}\Bigr]_q\,
(q^{(b-1)}v;q)_{i-a-b+1}\, ,\quad a,b=1,2,\dots ,i.
\ee
Here the superscript symbol $[i]$ in the notation  $J^{[i]}$ indicates the size of the matrix. In the right hand side of the formula (\ref{JJ}) we use standard notation for the $q$-Pochhammer symbol
\be
\lb{Poch}
(v;q)_m := \prod_{k=0}^{m-1} (1- q^k v),
\ee
and for the Gaussian binomial coefficients 
\be
\lb{binomial}
\Bigl[ { m\atop k}\Bigr]_q := {(q^{m-k+1};q)_{k}\over (q;q)_k} = {[m]_q!\over [k]_q!\, [m-k]_q!}, \qquad [i]_q!:= \prod_{k=1}^i [k]_q,\qquad [k]_q := {1-q^k\over 1-q}.
\ee
Since $[{m\atop k}]_q=0$ for $k>m$ all matrices $J^{[i]}$ turn out to be
upper anti-triangular.\smallskip

Based on our experience in constructing R-matrices in low-dimensional spaces, we formulate the following 
\begin{prop}
	\lb{prop1}
In search of an R-matrix satisfying the Ansatz (\ref{ansatz}), we define
its pre-equatiorial blocks $R^{(2)},\dots ,R^{(n+1)}$, $n<N$, by the formula
\be
\lb{pre-eq}
R^{(i+1)}_{a,b} = x^{(b-a)}J^{[i]}_{a,b}(q,v), \quad i=1,\dots , n.
\ee
Here $x\in {\Bbb C}\setminus \{0\}$ is the diagonal twist parameter and $q, v\in {\Bbb C}\setminus \{0\}$ are the essential parameters related to the eigenvalues of the R-matrix.

With this choice   the blocks $Y^{(3)}, \dots , Y^{(n+2)}$ of  the braid relation  vanish.
Equations in the block $Y^{(n+3)}$ then admit following solutions.
\ba
\lb{pre-eq-1}
\mbox{$a)$} && R^{(n+2)}_{a,b} = g_{a,b}\, x^{(b-a)}J^{[n+1]}_{a,b}(q,v),\hspace{29mm}
\ea
where $g_{1, a}=g\in  {\Bbb C}\setminus \{0\}$, $g_{a ,1 }=g^{-1}$ if $a\notin\{1,\, n+1\}$, and $g_{a,b}=1$ otherwise. Here $g$ is the gauge parameter of the R-matrix.

Under additional restriction on the parameters $q$, $v$
\be 
\lb{restrict}
\mbox{$b_1)$}\;\;1-q^{n-1} v=0, \;\;\; \mbox{or\;\;\; $b_2)$}\;\;  \mbox{$q$ is the
	primitive root of unity of degree\, $n$.}
\ee
there exists another solution
\ba
\lb{pre-eq-2a}
\mbox{$b)$} && R^{(n+2)}_{a,b} = x^{(b-a)}J^{[n+1]}_{ab}(q,v)\;\;\forall\, a,b\in\{2,3,\dots ,n\},
\\
\lb{pre-eq-2b}
&& R^{(n+2)}_{1,a} =R^{(n+2)}_{a,1} = 0, \;\; \forall\, a\in\{2,3,\dots ,n\},
\\
\lb{pre-eq-2c}
&& R^{(n+2)}_{1,1} =1-v',\quad R^{(n+2)}_{1,n+1} = x', \quad 
R^{(n+2)}_{n+1,1} = v'/x'.
\ea
In these formulas  $x'\in  {\Bbb C}\setminus \{0\}$ and $v'\in  {\Bbb C}\setminus \{0\}$ are additional diagonal twist  and  eigenvalue parameters, respectively.
\end{prop}

\begin{rem}
{\rm
Extending the solution (\ref{JJ}), (\ref{pre-eq}) to arbitrarily large values of the index $n$ we can view it as an R-matrix acting on the tensor square of the infinite-dimensional space $V$. This R-matrix reproduces  (up to a diagonal twist transformation) the R-matrix constructed in \cite{JK}, Theorem 7, by specializing of the  universal R-matrix of the algebra $U_q({\frak{sl}(2)})$ to the tensor square of Verma module. In \cite{Will}, this R-matrix was used to construct a unified knot invariant, 
also known as the two-variable colored Jones invariant \cite{H}.
}
\end{rem}

\noindent {\bf Proof.}~
The assertion that the preequatorial blocks (\ref{pre-eq}) satisfy the braid relations in blocks $Y^{(3)}, \dots , Y^{(n+2)}$ follows by induction on $n$. The base for induction is provided by the examples from the previous section. The inductive step will be proven once we verify case $a)$ of the proposition.\smallskip

We do not provide a detailed verification of the cases $a)$ and $b)$. Instead, we outline the main steps of the considerations.

The braid relations from the block $Y^{(n+3)}$ contain a triangular linear system of equations for the components of the matrix $R^{(n+2)}$ that do not belong to its first column and first row. This system has a unique solution if~\footnote{This condition is sufficient but not necessary.}
\be
\lb{cond}
(v;q)_{n-1}= (1-v)(1-q v)\dots (1-q^{n-2}v)\neq 0,
\ee
and its solution is given by  formula (\ref{pre-eq-2a}) which is common for these components of $R^{(n+2)}$ in both cases -- $a)$ and $b)$ -- in the proposition.

The conditions imposed on the components of the first column and first row of $R^{(n+2)}$ within block $Y^{(n+3)}$ are quadratic-linear. Some of them factorize, which finally gives two branches of the solution.
\begin{itemize}
	\item[1)] If at least one of the components $R^{(n+2)}_{1,a}$, $ R^{(n+2)}_{a,1}$, $a=2,\dots , n$, is different from zero, then
	\be 
	\lb{aga}
R^{(n+2)}_{1,n+1}=x^n, \;\;\; \mbox{and}\;\;\; R^{(n+2)}_{n+1,1}=x^{-n} v^n.
	\ee 
	\item[2)] Otherwise
	$
	R^{(n+2)}_{1,a}=R^{(n+2)}_{a,1}=0$, $\forall\, a=2,\dots ,n$,
	and there are no restrictions on the values of $R^{(n+2)}_{1,n+1}$ and $R^{(n+2)}_{n+1,1}$.
\end{itemize}

Branch $1)$ corresponds to the solution $a)$ in the proposition, In this case,  substituting (\ref{aga}) into the braid relations for the block $Y^{(n+3)}$ yields a system of linear homogeneous equations for the components $R^{(n+2)}_{1,a}$, $a=2,\dots ,n$. Under condition (\ref{cond}), the system has a one-parameter family of solutions, given by formula (\ref{pre-eq-1}), where $g$ is the parameter of this family. Assuming $g \neq 0$, the components  $R^{(n+2)}_{a,1}$, $a=2,\dots ,n$, and then, $R^{(n+2)}_{1,1}$ can be uniquely recovered from the remaining relations. The result is presented in formula (\ref{pre-eq-1}).

In the case $g=0$,  the braid relations in block $Y^{(n+3)}$ are consistent only under additional restrictions (\ref{restrict}) on the parameters $q$ and $v$.
In the following we will constrain ourselves to considering the case  $g\neq 0$.\footnote{Note that case $g=0$ with additional conditions (\ref{restrict}) can also be described by formula (\ref{pre-eq-1}). Case $b_1)$ in (\ref{restrict}) corresponds to the choice $g=1$ in (\ref{pre-eq-1}), whereas in case $b_2)$
one must set $g=[n]_q$ in (\ref{pre-eq-1}) before imposing the constraint (\ref{restrict}).}
\smallskip

Branch $2)$ corresponds to the solution $b)$ in the proposition, In this case, 
the system of braid relations in the block $Y^{(n+3)}$ is consistent only if additional constraints  (\ref{restrict})  are imposed on the parameters $q$ and $v$. In this scenario, the matrix $R^{(n+2)}$ takes a block-diagonal form with $2\times 2$ sub-block in the first and last rows and columns. 
Braid relations, together with the invertibility condition on $R$,  determine the following parameterization for the components of this block:\footnote{Choosing the $2 \times 2$ matrix that is identically zero also satisfies the braid relations. This solution will be used in the next section in the proof of proposition \ref{prop2}.}
$$
\left(
\begin{array}{cc}
	1-v' & x'
	\\
	v'/x' & 0
\end{array}
\right),  \qquad x', v' \in {\Bbb C}\setminus\{0\}.
$$
It is precisely these formulas that are given in (\ref{pre-eq-2c}).
\hfill$\blacksquare$\medskip

\section{R-matrices in finite dimensional spaces. }\lb{sec5}
\subsection{Type $GL(2)$ R-matrices  in space $V$ of any dimension} \lb{sec5a}

So far, we have discussed the pre-equatorial blocks of the R-matrix. Solving the braid relation for the post-equatorial blocks imposes additional restrictions on the eigenvalue parameters $q$ and $v$. The form of these conditions depends on the choice of solution branch from proposition  \ref{prop1}. Let us first consider branch a) (\ref{pre-eq-1}).
\begin{prop}
	\lb{prop2}
	Consider a conservative upper anti-triangular operator $R\in {\rm Aut}(V^{\otimes 2})$, $\mbox{dim}\, V=N$, whose pre-equatorial blocks have the form
	 \be
	 \lb{gl2-1}
	 R^{(i+1)}_{a,b} = J^{[i]}_{a,b}(q,v), \qquad\;\;\; i=1,\dots ,N, \qquad\;\, a,b=1,\dots ,i.
	 \ee
	 This operator admits a unique extension to an R-matrix in which the post-equatorial blocks are defined as follows:
	 \be
	 \lb{gl2-2}
	 R^{(N+i+1)}_{a,b} = J^{[N+i]}_{a+i,b+i}(q,v), \quad i=1,\dots ,N-1, \;\;\; a,b=1,\dots ,N-i.
	 \ee
and the parameters $q$ and $v$ are constrained by one of the following conditions:
\be 
\lb{restrict-2}
\mbox{$a_1)$}\;\;
1-q^{N-1} v=0, \;\;\; \mbox{or\;\;\; $a_2)$}\;\; 
	\mbox{\rm $q$ is the
		primitive root of unity of degree $N$.}
\ee	

Each matrix $R^{(i+1)}$, $i=1,\dots ,N$, has the set of eigenvalues
\be 
\lb{eigen-1}
\lambda_k := q^{k(k-1)\over 2}(-v)^k,\quad k=0,1,\dots,i-1.
\ee
 
In the case $a_1)$ blocks of the R-matrix are symmetric with respect to the equator
 \be
 \lb{symm-eq}
\left( R^{(N-i+1)}-R^{(N+i+1)}\right)\vert_{v=q^{1-N}}=0, \quad i=1,\dots , N-1,
 \ee
 and so, in a generic situation (i.e., if all eigenvalues $\lambda_k$ are pairwise different) the spectrum of $R$ in the case $a_1)$ is
 \be 
\lb{specR-a1}
{\rm Spec\,} R = \left\{  \lambda_k^{\# (2N-2k-1)}   \right\}_{k=0,\dots N-1}.
\ee
Here, the superscript symbol in the notation $\lambda^{\# m}$ indicates the multiplicity of the eigenvalue $\lambda$ in the spectrum of the operator $R$.

In the case $a_2)$ the post-equatorial blocks of the R-matrix have the following eigenvalues
\be
\lb{eigen-2}
{\rm Spec\,} R^{(N+i+1)} = \left\{  \lambda_k   \right\}_{k=i,i+1,\dots N-1}.
\ee
and so, for generic values of the parameter $v$ the spectrum of $R$ in the case $a_2)$ is
\be 
\lb{specR-a2}
{\rm Spec\,} R = \left\{  \lambda_k^{\# N}   \right\}_{k=0,\dots N-1}.
\ee
\end{prop}

\begin{rem}
	\lb{rem5.2}{\rm
The multiplicities of the eigenvalues in the spectra of the R-matrices   (\ref{specR-a1}) and (\ref{specR-a2}) coincide with the dimensions of the irreducible components in the decompositions of the tensor squares of the standard spin representations and the typical nilpotent representations of the algebra $U_q(\frak{sl}(2))$, respectively (see, e.g., \cite{A}, or \cite {RT}, eq.(7.4.1) and  \cite{GPT}, lemma 30).  This  clearly
indicates that the numerical R-matrices in cases $a_1)$ and $a_2)$ are related to realizations of the universal R-matrix of $U_q(\frak{sl}(2))$ 
on the tensor squares of the aforementioned representations. That is indeed the case.

In case $a_1)$, the R-matrix (\ref{gl2-1}), (\ref{gl2-2}) is equivalent
(up to simple gauge, diagonal twist and scaling transformations) 
to the R-matrix used in \cite{RT} (section 7.4) to define a remarkable family of link invariants --- coloured Jones polynomials.
In case $a_2)$, the R-matrix (\ref{gl2-1}), (\ref{gl2-2}) is equivalent (up to gauge, diagonal twist and scaling transformations)  to the R-matrices considered in \cite{ADO} and \cite{M2}, where they were used to construct the so-called ADO link invariants, also known as colored Alexander polynomials. 
Interrelations between the two-variable Jones invariant, coloured Jones invariants, ADO invariants, and also the Kashaev's quantum dilogarithm invariants of knots \cite{K} are described in \cite{MM,CGP,Will}.
}
\end{rem}

\begin{rem}{\rm
			The R-matrix described in proposition \ref{prop2} is a unique element of its gauge and  diagonal twist family. 
			It is distinguished by the property that the sum of the components of this R-matrix along any of its columns is equal to 1. In case $a_1)$, these constant R-matrices  are equivalent  to the limiting case of depending on the so-called {\em spectral parameter}\,\footnote{It should not be confused with the spectral value of the constant R-matrix.} R-matrices  constructed in \cite{Mang}. These R-matrices have recently applied for the construction of the integrable stochastic processes \cite{BP}. }
\end{rem}

\noindent {\bf Proof.}~
We can construct the post-equatorial blocks of the R-matrix, drawing on our experience with the pre-equatorial blocks. To this end, we consider embedding of the space $V$ into a space $W$ of twice the dimension, $\mbox{dim}\, W = 2 \mbox{dim}\, V$, choosing a linear basis in $W$ such that its first $N$ elements are the basis vectors of the space $V$. 
Using the lexicographical ordering of the basis vectors in $W^{\otimes 2}$, we define a linear operator $R_{ext} \in \mbox{End}(W^{\otimes 2})$ that coincides with $R$ when acting on the basis vectors of the space $V^{\otimes 2}$ (i.e., $R_{ext}|_{V^{\otimes 2}} = R$) and vanishes identically on the remaining basis vectors of the space $W^{\otimes 2}$. By construction, $R_{ext}$ is degree conserving and upper anti-triangular, and it satisfies the braid relations if and only if the operator $R$ satisfies them. Moreover, thanks to our choice $\mbox{dim}\, W = 2\mbox{dim}\, V$, all the post-equatorial blocks of the operator $R_{ext}$ vanish identically. 
Thus, the braid relations for the post-equatorial block of the operator $R$ -- $R^{(N+i+1)}$ -- take on the already familiar form of the braid relations for the pre-equatorial block $R_{ext}^{(N+i+1)}$ of the operator $R_{ext}$, subject to the additional condition that the first and last $i$ columns and rows of this block are identically zero.

This observation, along with the considerations in proof of the proposition \ref{prop1} (see the footnote at the end of the proof) justifies formula (\ref{gl2-2}) for the post-equatorial blocks of $R$. Conditions (\ref{restrict-2}) on the parameters $q$ and $v$ appear for the first time in consideration of the braid relations in the block $Y_{ext}^{(N+3)}$ and then, reappear when solving the subsequent blocks of braid relations  $Y_{ext}^{(N+i)}$, $i=4, \dots, 3N-1$.\smallskip

Relations (\ref{eigen-1})-(\ref{specR-a2}) concerning the spectra of the R-matrices were observed experimentally. Their proof is a straightforward, although lengthy, calculation.
\hfill$\blacksquare$

\subsection{Type $GL(2|1)$ R-matrices  in space $V$ of odd dimension
 } \lb{sec5.2}

\begin{prop}
\lb{prop3}	
Consider a conservative upper anti-triangular operator 
$R\in {\rm Aut}(V^{\otimes 2})$, where the space $V$ has odd dimension: $\mbox{dim}\, V =N= 2n-1$, $n\geq 2$, The blocks $R^{(i)}$,  $i=2,\dots ,2N$, of the operator $R$ are defined as follows.

The first $n$ blocks and the last $n-1$ blocks have the form
\ba
\lb{B1}
R^{(i+1)}_{a,b} &=& J^{[i]}_{a,b}(q,v), \qquad\qquad\;\qquad\;\;\mbox{for}\;\; i=1,\dots ,n, \qquad\;\, a,b=1,\dots ,i;
\\[2pt]
\lb{B2}
R^{(3n+i-1)}_{c,d} &=& -\,v\, J^{[n+i-2]}_{c+i-1,d+i-1}(q,q v), \quad\mbox{for}\;\; i=1,\dots ,n-1, \;\;\; c,d=1,\dots ,n-i;
\ea

The blocks  $R^{(n+i+1)}$, $i=1,\dots ,n-1$, contain two anti-triangular smaller blocks:
the central $(n-i)\times (n-i)$ sub-block $R_{C}^{(n+i+1)}$ with indices running from $i+1$ to $n$, and the border $2i\times 2i$ sub-block $R_{B}^{(n+i+1)}$ whose indices run from $1$ to $i$ and from $n+1$ to $n+i$. The latter sub-block contains 3 nontrivial $i\times i$ matrices
$R_{B_0}^{(n+i+1)}$, $R_{B_+}^{(n+i+1)}$ and $R_{B_-}^{(n+i+1)}$.
The structure of $R^{(n+i+1)}$ is shown in the figure below.\newpage

\ba
\nn
&\underline{\mbox{\footnotesize\rm indices:}\hspace{9mm}
\mbox{\footnotesize\rm 1}\qquad\;\dots\qquad\;\; 
\mbox{\footnotesize\rm i}\quad \mbox{\footnotesize\rm i+1}\quad\;\;\; \dots\quad\;\;\; 
\mbox{\footnotesize\rm n}\;\;\; \mbox{\footnotesize\rm n+1}\quad\;\;\dots\quad\;\; 
\mbox{\footnotesize\rm n+i}}
&
\\[-4pt]
\lb{pic}
&
R^{(n+i+1)}=\left(\begin{array}{cccc|cccc|cccc}
	.&.&.&\hspace{3.5mm}.&&&&&.&.&.&\hspace{3.5mm}.\\
	.&\hspace{-2mm}R_{B_0}^{(n+i+1)}\hspace{-3mm}&.&\hspace{3.5mm} .&&\mbox{\raisebox{-3mm}{\hspace{7mm}\Large\rm 0}}&&&.&\hspace{-2mm}R_{B_+}^{(n+i+1)}\hspace{-3mm}&.&\\[-1.5mm]
	.&.&.&&&&&&.&.&&\\
	.&.&&&&&&&.&&&\\
	\hline
	&&&&.&.&.&\hspace{3.5mm}.&&&\\
	&\mbox{\raisebox{-3mm}{\hspace{6mm}\Large\rm 0}}&&&.&\hspace{-2mm}R_{C}^{(n+i+1)}\hspace{-3mm}&.&&&\mbox{\raisebox{-3mm}{\hspace{8mm}\Large\rm 0}}&&\\[-1.5mm]
	&&&&.&.&&&&&&\\
	&&&&.&&&&&&&\\
	\hline
	.&.&.&\hspace{3.5mm}.&&&&&&&&\\
	.&\hspace{-2mm}R_{B_-}^{(n+i+1)}\hspace{-3mm}&.&&&\mbox{\raisebox{-3mm}{\hspace{7mm}\Large\rm 0}}&&&&\mbox{\raisebox{-3mm}{\hspace{8mm}\Large\rm 0}}&&\\[-1.5mm]
	.&.&&&&&&&&&&\\
	.&&&&&&&&&&&
\end{array}\right), 
&
\ea
Here, the  entirely zero blocks of the matrix $R^{(n+i+1)}$ are marked with $0$, while the non-zero components of the non-trivial blocks are indicated by dots. The matrices
$R_{C}^{(n+i+1)}$, $R_{B_+}^{(n+i+1)}$ and $R_{B_-}^{(n+i+1)}$ are anti-triangular, whereas in the matrix $R_{B_0}^{(n+i+1)}$ the components adjacent to the anti-diagonal from below are also non-zero.  The nontrivial sub-blocks in (\ref{pic}) have the form
\ba
\lb{B3}
\left(R_{C}^{(n+i+1)}\right)_{a,b} &=& J^{[n+i]}_{a+i,b+i}(q,v), \qquad \qquad \qquad \quad\;\;\; a,b=1,\dots , n-i,
\\[2pt]
\lb{B4}
\left(R_{B_+}^{(n+i+1)}\right)_{c,d} &=& J^{[i]}_{c,d}(q,qv), \qquad
\left(R_{B_-}^{(n+i+1)}\right)_{c,d} \,=\, v\, J^{[i]}_{c,d}(q,v),
\\[2pt]
\lb{B6}
\left(R_{B_0}^{(n+i+1)}\right)_{c,d} &=&  J^{[i+1]}_{c,d}(q,v) -  v\, J^{[i]}_{c-1,d}(q,v),\qquad\qquad  c,d= 1,\dots , i,
\ea
where indices of all the sub-matrices in (\ref{pic}) are enumerated starting from 1, and we assume $J^{[i]}_{0,d}=0$.\smallskip

The blocks  $R^{(2n+i)}$, $i=1,\dots ,n-1$, have a structure analogous to that shown in  figure (\ref{pic}). They contain two anti-triangular smaller blocks:
the central $(i-1)\times (i-1)$ sub-block $R_{C}^{(2n+i)}$,\footnote{This block is absent in case $i=1$.} and the border $2(n-i)\times 2(n-i)$ sub-block $R_{B}^{(2n+i)}$, which contains 3 nontrivial $(n-i)\times (n-i)$ matrices 
$R_{B_0}^{(2n+i)}$, $R_{B_+}^{(2n+i)}$ and $R_{B_-}^{(2n+i)}$. All matrices
$R_{C}^{(2n+i)}$, $R_{B_+}^{(2n+i)}$, $R_{B_-}^{(2n+i)}$  and $R_{B_0}^{(2n+i)}$ are now anti-triangular.
They are defined as
\ba
\lb{B7}
\left(R_{C}^{(2n+i)}\right)_{a,b} &=& -v\, J^{[i-1]}_{a,b}(q,q v), \qquad \qquad \qquad \qquad\qquad a,b=1,\dots , i-1,
\\[2pt]
\lb{B8}
\left(R_{B_+}^{(2n+i)}\right)_{c,d} &=& J^{[n+i-1]}_{c+i,d+i-1}(q,qv), \qquad
\left(R_{B_-}^{(2n+i)}\right)_{c,d} \,=\, v\, J^{[n+i-1]}_{c+i-1,d+i}(q,v),
\\[2pt]
\lb{B9}
\left(R_{B_0}^{(2n+i)}\right)_{c,d} &=&  J^{[n+i-1]}_{c+i,d+i-1}(q,q v) -  v\, J^{[n+i-1]}_{c+i-1,d+i}(q,v),\quad\;  c,d= 1,\dots , n-i,
\ea
This operator becomes the R-matrix if the parameters $q$ and $v$ satisfy the condition	
\be 
\lb{restrict-B}
1-q^{n-1} v=0.
\ee	
Then, the spectra of the blocks of the R-matrix are following
\ba
\lb{specR-b1}
&
{\rm Spec\,} R^{(i+1)} = \left\{  \lambda_0,\dots ,\lambda_{i-1}  \right\}, \;\; {\scriptstyle i=1,\dots ,n,}
&
\\[2pt]
&
\lb{specR-b2}
{\rm Spec\,} R_{C}^{(n+i+1)} = \left\{  \lambda_0,\dots ,
\lambda_{n-i-1}   \right\},\; 
{\rm Spec\,} R_{B}^{(n+i+1)} = \left\{  \lambda_0,\lambda_1^{\# 2},\dots ,   \lambda_{i-1}^{\# 2},\lambda_i\right\}, \;\; {\scriptstyle i=1,\dots ,n-1,}
\\[2pt]
\lb{specR-b3}
&
{\rm Spec\,} R_{C}^{(2n+i)} = \left\{  \lambda_1,\dots ,
\lambda_{i-1}   \right\},\; 
{\rm Spec\,} R_{B}^{(2n+i)} = \left\{  \lambda_0,\lambda_1^{\# 2},\dots ,   \lambda_{n-i-1}^{\# 2},\lambda_{n-i}\right\}, \;\; {\scriptstyle i=1,\dots ,n-1,}
&
\\[2pt]
\lb{specR-b4}
&
{\rm Spec\,} R^{(3n+i-1)} = \left\{  \lambda_1,\dots ,\lambda_{n-i}  \right\}, \;\; {\scriptstyle i=1,\dots ,n-1,}
&
\ea
where the eigenvalues $\lambda_k=(-1)^k q^{k(k-N)\over 2}$, $k=0,\dots ,n-1$, are given by eq. (\ref{eigen-1}). Hence, for generic values of the parameter $q$ the spectrum of R-matrix (\ref{B1})--(\ref{B9}) is
\be 
\lb{specR-B}
{\rm Spec\,} R = \left\{ \lambda_0^{\#(2N-1)}\right\}\cup
\left\{ \lambda_k^{\# 4(N-2k)}   \right\}_{k=1,\dots n-1}, \quad (N=2n-1,\; n\geq 2).
\ee
\end{prop}

\begin{rem}{\rm
		The multiplicities of the eigenvalues in the spectra of the R-matrices   (\ref{B1})--(\ref{B9}) coincide with the dimensions of the irreducible components in the decompositions of the tensor square of the symmetric tensor representations of the Lie super-algebra $\frak{sl}(2|1)$ (see \cite{BR} and \cite{GQS}).  
		Therefore, we believe that these numerical R-matrices  are  related (up to rescaling and diagonal twist transformation) to realization of the universal R-matrix (see \cite{KT}, \cite{Y}) of the quantized universal enveloping algebra $U_q(\frak{sl}(2|1))$ on the tensor squares of its symmetric tensor representations. That is why we call it the R-matrix of the type $GL(2|1)$.
	}
\end{rem}
\begin{rem}{\rm
In \cite{DKK}, a general formula for the $\frak{sl}(2\vert1)$-invariant R-matrix was obtained (for further developments, see \cite{DG} and the references therein).
Unlike our constant R-matrices (\ref{B1})--(\ref{B9}),  
 this R-matrix depends on the spectral parameter and acts on the tensor product of infinite-dimensional spaces; however, it corresponds to a specific "classical" choice of our parameter $q$: $q=1$. 
We believe that the R-matrix derived in \cite{DKK} and the R-matrices from proposition \ref{prop3}	represent different limits of the same $U_q(\frak{sl}(2\vert1))$-invariant R-matrix. We are going to trace the connection between them in a forthcoming paper devoted to the baxterization of the constant R-matrices obtained in the present work.
}
\end{rem}
 
\begin{rem}{\rm
	One might hope that applying the Reshetikhin–Turaev procedure \cite{RT,RT2} to the R-matrices (\ref{B1})–(\ref{B9}) would yield meaningful link invariants.
	To our surprise, this is not the case.
	Direct calculations for R-matrices acting on spaces $V$ of low dimensions (3, 5, and 7) --- performed for knots and links with a small number of crossings --- yield trivial invariants identically equal to 1.
	We conjecture that this triviality holds for all R-matrices in this family and for arbitrary links.
		}
\end{rem}

\begin{rem}{\rm
		Unlike the R-matrices (\ref{gl2-1}) and (\ref{gl2-2}), the R-matrix defined by (\ref{B1})–(\ref{B9}) cannot be transformed—via gauge transformations and a diagonal twist—into a form suitable for stochastic applications. Note, however, that for the R-matrices in this series, the column sums take on only two distinct values. For the blocks $R^{(i)}$ with $i=2,\dots,2n$, the sum of elements in each column is $1$. For $i=3n,\dots,4n-2$, the column sums equal $-v=-1/q^{n-1}$. Finally, for $i=2n+1,\dots,3n-1$, the column sums are $-v$ for the central sub-block (\ref{B7}) and $1$ for the boundary sub-block defined by formulas (\ref{B8}) and (\ref{B9}).}
\end{rem}

\noindent {\bf Proof.}~
The R-matrices presented in proposition \ref{prop3} originate from sequences of the pre-equatorial blocks in proposition \ref{prop1}, branch b), see eqs. (\ref{pre-eq-2a})-(\ref{pre-eq-2c}).
These blocks satisfy the braid relations in blocks $Y^{(3)}, \dots , Y^{(n+3)}$ under restrictions (\ref{restrict}) on their parameters $q$, $v$. 
The check shows that the braid relations in the higher blocks $Y^{(i)}$, $i>n+3$,  are inconsistent 
if the dimension of  space $V$ is less than $2n - 1$. Below we consider the case 
$\mbox{dim}\, V= 2n-1$. 
We successively solve braid relations for the R-matrix blocks
$R^{(n+3)}, R^{(n+4)}, \dots , R^{(4n-2)}$. The Ansatz ensures that the equations for their components are either quadratic or linear. As in the proof of proposition 
\ref{prop1}, we will explicitly consider  only solutions to the quadratic relations, whereas for the systems of linear relations, we will merely present their solutions.
\smallskip

Consider braid relations in block $Y^{(n+4)}$. They impose following conditions on the components of the R-matrix block $R^{(n+3)}$:
\begin{enumerate}
	\item[-]
	$R^{(n+3)}$ has the block-diagonal form with two blocks: a boundary sub-block $R_{B}^{(n+3)}$ of size 4×4 located in the first two and last two rows and columns, and a central 
	sub-block $R_{C}^{(n+3)}$ of size $(n-2)\times(n-2)$.
	\item[-] Components of  the central sub-block have a standard form (c.f. with eq. (\ref{gl2-2}))
	$$
	\left(R_{C}^{(n+3)}\right)_{a,b} = J^{[n+2]}_{a+2,b+2}(q,v), \quad a,b =1,\dots , n-2.
	$$ 
	\item[-] Assuming that the  border block $R_{B}^{(n+3)}$ does not split further into diagonal sub-blocks, i.e., that at least one of components $R^{(n+3)}_{\;a,1}$, $R^{(n+3)}_{\;1,a}$, $a\in\{2,n+1\}$ is  non-zero, one has
	\be
	\lb{branch-2} 
	R^{(n+3)}_{1,n+2}=x', \quad R^{(n+3)}_{n+2,1}={v^2/ x'}, \quad (v-v')R^{(n+3)}_{\;2,2}=0.
	\ee
\end{enumerate}
The last of these conditions gives rise to two branches in the solution of the braid relations. Here, we analyze the case $v=v'$. The case $R^{(n+3)}_{\;2,2}=0$ will be considered  in the next subsection.

The braid relarions in block $Y^{(n+4)}$, and in all higher blocks $Y^{(k)}$, $k> n+4$,  do not impose any conditions on the components $R^{(n+3)}_{1,n+1}$ and $R^{(n+3)}_{2,n+1}$, 
and we redenote them briefly
$$
R^{(n+3)}_{1,n+1}=g'', \quad R^{(n+3)}_{2,n+1}=x''.
$$
Assuming additionally that $g''$ and $x''$ are both different from $0$,
we can interpret them, respectively, as the new gauge and diagonal twist R-matrix parameters.

Furthermore, the braid relations in block $Y^{(n+4)}$ 
involve a quadratic condition on $R^{(n+3)}_{\;1,1}$ and $R^{(n+3)}_{\;2,1}$ that factors into linear terms
\be
\lb{branch-3}
\left(x'' R^{(n+3)}_{\;1,1}- g'' R^{(n+3)}_{\;2,1}\right)\left(x'' R^{(n+3)}_{\;1,1}-g'' R^{(n+3)}_{\;2,1}+x''(1-v^2)\right)\,=\,0.
\ee
As it turns out, the branch of solution
corresponding to setting the second factor in eq.(\ref{branch-3}) to zero is a dead end. 
By setting the first factor in (\ref{branch-3}) to zero, we finally
find a one parametric family of solutions for the border block $R_{B}^{(n+3)}$.
By parameterizing this family with a new eigenvalue $v''$ of the boundary block, we obtain the following expression for $R^{(n+3)}$
\ba
\lb{pic2}
&
R^{(n+3)}=\left(\begin{array}{cc|c|cc}
	(1-v)(1-{v''\over v})&{g'' v''\over x'' v}& & g''&x'\\[3pt]
	(1-v)(1-{v''\over v}){x''\over g''}&{v''\over v}-v&\raisebox{3mm}{\Large 0}& x''&0\\[4pt]
	\hline
\hspace{2cm}\raisebox{-2mm}{\Large 0 }	&&\raisebox{-2mm}{$||J^{[n+2]}(q,v)||_{a,b=3}^n$}&
\hspace{8mm}\raisebox{-2mm}{\Large 0 }\hspace{-2mm}&
	\\[9pt]
	\hline
	{(1-v)(v-v'')\over g''}	& \;\;\;\;{v''\over x''}^{^{\phantom{J^J}}}& &&
	\\[4pt]
	{v^2\over x'}	&0&\raisebox{4mm}{\Large 0}&\hspace{8mm}\raisebox{4mm}{\Large 0}&
\end{array}\right).
&
\ea
Note that the spectrum of the border sub-block has the form
$$
\mbox{\rm Spec\,}R_{B}^{(n+3)}\, =\,\{1,-v^{\# 2},v''\}.
$$

Next, we proceed to consider the braid relations in block $Y^{(n+5)}$. They dictate the block-diagonal structure for $R^{(n+4)}$, with the central sub-block of a size $(n-3)\times (n-3)$ and the border $6\times 6$ sub-block. Components of the central block are fixed in the standard way
$$
\left(R_{C}^{(n+4)}\right)_{a,b} = J^{[n+3]}_{a+3,b+3}(q,v), \quad a,b =1,\dots , n-3.
$$
The components of the border block are uniquely determined if we assume that its outermost rows and columns do not separate into a diagonal 
$2\times 2$ sub-block and fix the penultimate element of its first row, assuming it is non-zero:
$$
\left(R_{B}^{(n+4)}\right)_{1,5}\equiv R^{(n+4)}_{1,n+2}=g''',\quad g'''\neq 0.
$$
The new parameter $g'''$ is a gauge parameter of the R-matrix. 

An important feature of this stage is the appearance, among the relations of block $Y^{(n+5)}$, of a condition on the  spectral values of the R-matrix: 
\be
\lb{cond-spec}
v''=q v^2.
\ee
Taking into account this condition, an explicit expression for the border block in the case $x'=x''=g'=g''=g'''=1$ is given by the formulas (\ref{B3})-(\ref{B6}). Expression for 
$R_{B}^{(n+4)}$ 
with arbitrary values of the diagonal twist parameters  $x$ and $x'$  are given  in the appendix.

The subsequent solution of the braid relations in blocks  $Y^{(n+2+i)}$, $i\leq n-1$, which contain 
only pre-equatorial blocks of the R-matrix, proceeds in a similar manner. 
Resolving  the braid relations in block $Y^{(n+2+i)}$, 
determines block $R^{(n+1+i)}$ of the R-matrix.
It has a block-diagonal form with two sub-blocks ---
a central one and a border one --- of sizes $(n-i)\times(n-i)$ and $2i\times 2i$, respectively.
Restricting consideration to the case where no diagonal $2\times 2$ sub-block with outermost columns and rows is separated in the border block,
we fix block $R^{(n+1+i)}$ up to a single gauge parameter 
$R^{(n+1+i)}_{1, n+i-1}=g^{(i)}$, which we assume to be non-zero.\smallskip

As we have already noted, the system of braid relations for the first post-equatorial block becomes consistent for the first time in space $V$ of dimension $2n-1$. The first post-equatorial block of the R-matrix ---  $R^{(2n+1)}$ --- is uniquely determined by the braid relations in block 
$Y^{(2n+2)}$. It is given by formulas (\ref{B7})-(\ref{B9}) for $i=1$, so that its central sub-block has zero size.
Furthermore, at this stage, one of the two solution branches (\ref{restrict}) for the parameters $q$ and $v$ is eliminated.
Only the possibility $b_1)$ remains: $1-q^{n-1} v=0$.\smallskip

Subsequently, resolution of the braid relations in each block $Y^{(2n+2+i)}$, $i=1,\dots ,n-2$, determines the block $R^{(2n+1+i)}$ of the R-matrix, consisting of two diagonal sub-blocks --- $R_{C}^{(2n+1+i)}$ and $R_{B}^{(2n+1+i)}$ --- of non-zero sizes $i\times i$ and $2(n-1-i)\times2(n-1-i)$, respectively.
The only exception to this rule is that the relations in the block $Y^{(2n+3)}$ do not uniquely determine the central component 
in block $R^{(2n+2)}$. Instead, they yield a quadratic condition for it:
$$
\left(R^{(2n+2)}_{n-1,n-1}+v\right)\left(R^{(2n+2)}_{n-1,n-1}-1\right)\, =\, 0.
$$
However, at the next stage, it turns out that the relations in block $Y^{(2n+4)}$ are consistent only  with the choice $R^{(2n+2)}_{\;n,n}=-v$. 

Formulas for the R-matrix blocks $R^{(2n+1+i)}$, $i=1,\dots ,n-2$, are given in (\ref{B7})-(\ref{B9}).\smallskip

Finally, the remaining blocks of the R-matrix $R^{(3n+i)}$, $i=0,\dots ,n-2$, have an anti-triangular form (\ref{B2}) and are uniquely determined by the conditions $Y^{(3n+1+i)}=0$.

Formulas (\ref{specR-b1})-(\ref{specR-B}) describing the spectra of the R-matrices were observed experimentally. 
\hfill$\blacksquare$\medskip

\subsection{Yet another sequence of R-matrices at the roots of unity in space $V$ of even dimension} \lb{sec5c}

\begin{prop}
	\lb{prop4}	Consider a conservative upper anti-triangular operator 
	$R\in {\rm Aut}(V^{\otimes 2})$, where the space $V$ has even dimension: $\mbox{dim}\, V =N= 2n$, $n\geq 2$, The blocks $R^{(i)}$,  $i=2,\dots ,2N$, of the operator $R$ are defined as follows.
	
	The first $n$ blocks and the last $n-1$ blocks have the form
	\ba
	\lb{C1}
	R^{(i+1)}_{a,b} &=& J^{[i]}_{a,b}(q,v), \qquad\qquad\;\qquad\;\;\mbox{for}\;\; i=1,\dots ,n, \qquad\;\, a,b=1,\dots ,i;
	\\[2pt]
	\lb{C2}
	R^{(3n+i+1)}_{c,d} &=& -\,{q\, v^{-1}} J^{[n+i]}_{c+i,d+i}(q,{q^{-1}v}), \quad\mbox{for}\;\; i=0,\dots ,n-1, \;\;\; c,d=1,\dots ,n-i;
	\ea
	
	The blocks  $R^{(n+i+1)}$, $i=1,\dots ,n$,  have a structure similar to that shown in  figure (\ref{pic}). They contain two anti-triangular smaller blocks:
	the central $(n-i)\times (n-i)$ sub-block $R_{C}^{(n+i+1)}$, and the border $2i\times 2i$ sub-block $R_{B}^{(n+i+1)}$ 
	containing 3 nontrivial $i\times i$ matrices
	$R_{B_0}^{(n+i+1)}$, $R_{B_+}^{(n+i+1)}$ and $R_{B_-}^{(n+i+1)}$.
	All four non-vanishing sub-blocks are anti-triangular. They have the form
	\ba
	\lb{C3}
	\left(R_{C}^{(n+i+1)}\right)_{a,b} &=& J^{[n+i]}_{a+i,b+i}(q,v), \qquad \qquad \qquad \quad\;\;\; a,b=1,\dots , n-i,
	\\[2pt]
	\lb{C4}
	\left(R_{B_+}^{(n+i+1)}\right)_{c,d} &=& J^{[i]}_{c,d}(q,q^{-1}v), \qquad
	\left(R_{B_-}^{(n+i+1)}\right)_{c,d} \,=\, q\,v^{-1} J^{[i]}_{c,d}(q,v),
	\\[2pt]
	\lb{C6}
	\left(R_{B_0}^{(n+i+1)}\right)_{c,d} &=&  J^{[i]}_{c,d}(q,v) -  q\,v^{-1} J^{[i]}_{c,d}(q,q^{-1}v),\qquad\;\;\;  c,d= 1,\dots , i.
	\ea
	
	The blocks  $R^{(2n+i+1)}$, $i=1,\dots ,n-1$, have a structure similar to (\ref{pic}) as well. They contain two anti-triangular smaller blocks:
	the central $i\times i$ sub-block $R_{C}^{(2n+i+1)}$, and the border $2(n-i)\times 2(n-i)$ sub-block $R_{B}^{(2n+i+1)}$ containing 3 nontrivial $(n-i)\times (n-i)$ matrices 
	$R_{B_0}^{(2n+i+1)}$, $R_{B_+}^{(2n+i+1)}$ and $R_{B_-}^{(2n+i+1)}$. All four non-vanishing sub-blocks are anti-triangular. They have the form
	\ba
	\lb{C7}
	\left(R_{C}^{(2n+i+1)}\right)_{a,b} &=& -q\,v^{-1} J^{[i]}_{a,b}(q,q^{-1} v), \qquad \qquad \qquad \qquad\qquad\;\; a,b=1,\dots , i,
	\\[2pt]
	\lb{C8}
	\left(R_{B_+}^{(2n+i+1)}\right)_{c,d} &=& J^{[n+i]}_{c+i,d+i}(q,q^{-1}v), \qquad
	\left(R_{B_-}^{(2n+i+1)}\right)_{c,d} \,=\, q\,v^{-1} J^{[n+i]}_{c+i,d+i}(q,v),
	\\[2pt]
	\lb{C9}
	\left(R_{B_0}^{(2n+i+1)}\right)_{c,d} &=&  J^{[n+i]}_{c+i,d+i}(q, v) - q\, v^{-1} J^{[n+i]}_{c+i,d+i}(q,q^{-1}v),\quad\;  c,d= 1,\dots , n-i,
	\ea
	This operator becomes the R-matrix if the parameter $q$  satisfies the condition	
	\be 
	\lb{restrict-C}
\mbox{\rm $q$ is the primitive root of unity of degree $n$.}
	\ee	
	Then, the spectra of the blocks of the R-matrix are following
	\ba
	\lb{specR-c1}
	&
	{\rm Spec\,} R^{(i+1)} = \left\{  \lambda_0,\dots ,\lambda_{i-1}  \right\}, \;\; {\scriptstyle i=1,\dots ,n,}
	&
	\\[2pt]
		\lb{specR-c2}
	&
	{\rm Spec\,} R_{C}^{(n+i+1)} = \left\{  \lambda_i,\dots ,
	\lambda_{n-1}   \right\},\; 
	{\rm Spec\,} R_{B}^{(n+i+1)} = \left\{  \lambda_{-1},\lambda_0^{\# 2},\dots ,   \lambda_{i-2}^{\# 2},\lambda_{i-1}\right\}, \;\; {\scriptstyle i=1,\dots ,n,}
	\\[2pt]
		\lb{specR-c3}
	&
	{\rm Spec\,} R_{C}^{(2n+i+1)} = \left\{  \lambda_{-1},\dots ,
	\lambda_{i-2}   \right\},\; 
	{\rm Spec\,} R_{B}^{(2n+i+1)} = \left\{  \lambda_0,\lambda_1^{\# 2},\dots ,   \lambda_{n-i-1}^{\# 2},\lambda_{n-i}\right\}, \;\; {\scriptstyle i=1,\dots ,n-1,}
	\qquad&
	\\[2pt]
		\lb{specR-c4}
	&
	{\rm Spec\,} R^{(3n+i)} = \left\{  \lambda_{i-2},\dots ,\lambda_{n-2}  \right\}, \;\; {\scriptstyle i=1,\dots ,n,}
	&
	\ea
	where the eigenvalues $\lambda_0, \dots, \lambda_{n-1}$ are given by eq. (\ref{eigen-1}), and $\lambda_{-1}=-q v^{-1}$. Hence, for generic values of the parameter $v$ the spectrum of R-matrix (\ref{C1})--(\ref{C9}) is
	\be 
	\lb{specR-C}
	{\rm Spec\,} R = \left\{ \lambda_{-1}^{\# N},\lambda_{n-1}^{\# N }\right\}\cup
	\left\{ \lambda_k^{\# 2N}   \right\}_{k=0,\dots n-2}, \quad (N=2n,\; n\geq 2).
	\ee
	\end{prop}
	
\begin{rem}{\rm
We have not established a connection between the multiplicities of eigenvalues in
the spectra of the R-matrices (\ref{C1})–(\ref{C9}) and the dimensions of the irreducible components in the fusion rules for representations  of a certain quantum group. We expect  that this R-matrix is related to representations of
$U_q(\frak{sl}(2|1))$ when $q$ is a roots of unity (see, e.g., \cite{AAB}).
		}
\end{rem}

\begin{rem}{\rm
		The standard Reshetikhin-Turaev procedure for constructing link invariants yields trivial, zero results when applied to 
		the family of $R$-matrices (\ref{C1})–(\ref{C9}).
		This is caused by the vanishing of the "quantum dimension" of the spaces $V$ determined by these R-matrices. One way to overcome this obstacle -- at least in the case of knots -- is to compute invariants of long knots \cite{K2} (for another approach, see \cite{GPT}).
		Our calculations of long-knot invariants for the $R$-matrices (\ref{C1})–(\ref{C9}) yield scalar operators whose scalar factors reproduce the classical Alexander polynomial $\Delta_K(t)$ (see, e.g., \cite{L}), up to a change of variable depending solely on the dimension of the underlying vector space $V$. We conjecture that for $\mbox{dim}\, V = 2n$, these factors coincide with $\Delta_K(v^n)$. We have directly verified the conjecture for $n=2, 3, 4$.}
\end{rem}

\noindent {\bf Proof.}~
The R-matrices presented in this and the previous propositions coincide until the block $R^{(n+2)}$.
The differences between them arise in the block $R^{(n+3)}$ and stem from the fact that the braid relation from the block $Y^{(n+4)}$ -- the rightmost in (\ref{branch-2}) --  admit two solutions.
Here we analyze the solution
$$
R^{(n+2)}_{\;2,2} =0.
$$

It should be noted that in this case, the minimum dimension of the space $V$ in which the braid relations are consistent turns out to be
$\mbox{dim}\, V = 2n$.
Choosing, as before, the notation
$$
R^{(n+3)}_{1,n+1}=g'', \quad R^{(n+3)}_{2,n+1}=x'',
$$
and assuming $g''\neq 0$, $x''\neq 0$, we use the braid relations in block $Y^{(n+4)}$ to fix all other components of $R^{(n+3)}$, with the exception of $R^{(n+3)}_{n+1,2}$, which satisfies a quadratic condition
$$
\left(x''R^{(n+3)}_{n+1,2}-1\right)\left(x''R^{(n+3)}_{n+1,2}-v v'\right)\, =\, 0.
$$
From this point on, we have two branches for the further solution of the braid relations:
\be
\lb{branch-22} 
\mbox{$c_1)$}\;\;
R^{(n+3)}_{n+1,2}=1/x'', \;\; v v'\neq 1,\;\;\; \mbox{or\;\;\; $c_2)$}\;\; 
R^{(n+3)}_{n+1,2}=v v'/x''.
\ee
In both cases, the construction of the pre-equatorial blocks of the R-matrix
$R^{(n+i+1)}$, $i=1,\dots ,n-1$,
proceeds similarly and the solutions are unique up to diagonal gauge transformations.
Their structure is analogous to that shown in figure (\ref{pic}), with all their non-trivial sub-blocks being 
upper anti-triangular.
The only difference between these two branches arises
when resolving the braid relations in block $Y^{(n+5)}$: in case $c_1)$, an additional condition on the spectral values of the R-matrix emerges  
$$
v'=q/v,
$$
whereas in case $c_2)$, no such restrictions appear.\smallskip

The subsequent construction of the equatorial and post-equatorial blocks of the R-matrix proceeds analogously to the constructions in proposition \ref{prop3}, with minor modifications that we discuss below.

When resolving the braid relations in block $Y^{(2n+3)}$, a quadratic condition arises on the component $R^{(n+2)}_{\; n,n}$ of the R-matrix:
\be
\lb{branch-4}
\left(R^{(2n+2)}_{\;n,n}-1\right)\left(R^{(2n+2)}_{\;n,n}- v'\right)\, =\, 0.
\ee
In case $c_2)$, both solution branches for this condition yield R-matrices that are tensor products of two known R-matrices: the $GL(2)$-type R-matrix from Proposition \ref{prop2}, acting on the tensor square of an $n$-dimensional space, and the R-matrices     
(\ref{dim2-R})  acting on the tensor square of a two-dimensional space. Thus, variant $c_2)$ does not yield any new, interesting R-matrices.

In case $c_1)$, the solution $R^{(2n+2)}_{\;n,n}=1$ of equation (\ref{branch-4}) leads to incompatible braid relations in block
$Y^{(2n+4)}$. The solution $R^{(2n+2)}_{\;n,n}=v'$ results in the elimination one of the solution branches (\ref{restrict}) when resolving the braid relations in block $Y^{(2n+3)}$. In this case, only the roots of unity branch (\ref{restrict-C}) remains.
Further analysis of this branch presents no difficulty, and we finally obtain the R-matrix (\ref{C1})-(\ref{restrict-C}).
\smallskip

Again, formulas (\ref{specR-c1})-(\ref{specR-C}) describing the spectra of the R-matrices were checked experimentally.  
\phantom{a}\hfill$\blacksquare$

\section*{Acknowledgments}
The authors thank Oleg Ogievetsky, Alexander Povolotsky and Sergey Derkachov for their interest in the work and for valuable comments.

\section*{Funding}
This work was carried out within the framework of the Basic Research
Program at HSE University (HSE-BR-2025-84).
\bigskip

\appendix
\section{Diagonal twist transformations of the R-matrices}
\lb{appendix-A}

\def\theequation{\thesection.\arabic{equation}}
\makeatletter\@addtoreset{equation}{section}\makeatother

In this appendix, we present formulas for families of R-matrices obtained via diagonal twist transformations \cite{Resh} from the R-matrices considered in Section \ref{sec5}. Note that these R-matrices are effectively obtained through the successive application of two twists: the first is a conjugation by the permutation operator, while the second is a conjugation by the composition of the permutation and a diagonal operator. The application of a pair of twist transformations is necessary because each twist violates our Ansatz by mapping the upper anti-triangular blocks of the R-matrix to the lower anti-triangular ones.
We provide answers without proofs, which consist of a direct verification of the braid relations.\smallskip

Let us introduce notation for diagonal twist matrix of size $i\times i$ depending on a single parameter $x\in {\Bbb C}\setminus \{0\}$:
$$
D_x^{[i]} := \mbox{diag}\{1,x,\dots , x^{i-1}\}.
$$
\subsection{$GL(2)$ type R-matrix, $\mbox{dim}\, V=N$}

A one-parametric family of the diagonal twist transformations of the $GL(2)$ type R-matrix (\ref{gl2-1}), (\ref{gl2-2}) is given on its blocks as follows 
\be
\lb{twist-A}
\begin{array}{lclr}
	R^{(i+1)}&\mapsto& \left(D_x^{[i]}\right)^{-1}\! R^{(i+1)}\, D_x^{[i]},& i=1,\dots , N;
	\\[12pt]
	R^{(N+i+1)}&\mapsto& \left(D_{x}^{[N-i]}\right)^{-1}\! R^{(N+i+1)}\, D_{x}^{[N-i]},& i=1,\dots , N-1.
\end{array}
\ee

\subsection{$GL(2|1)$ type R-matrix, $\mbox{dim}\, V=2n-1$}
A family of the diagonal twist transformations of the $GL(2|1)$ type R-matrix (\ref{B1})-(\ref{B9}), parameterized by three independent parameterts $x,y, z\in {\Bbb C}\!\setminus\! \{0\}$, is given on its blocks as follows
\be
\lb{twist-B}
\begin{array}{lclr}
R^{(i+1)}&\mapsto& \left(D_x^{[i]}\right)^{-1}\! R^{(i+1)}\, D_x^{[i]},& i=1,\dots , n;
\\[12pt]
R_{C}^{(n+i+1)}&\mapsto&  \left(D_{x}^{[n-i]}\right)^{-1}\!  R_{C}^{(n+i+1)}\, D_{x}^{[n-i]}, & 
\\[5pt]
R_{B_0}^{(n+i+1)} & \mapsto&   \left(D_{y}^{[i]}\right)^{-1}\! R_{B_0}^{(n+i+1)}\; D_{y}^{[i]}, &
\\[5pt]
R_{B_+}^{(n+i+1)} &\mapsto &  {\textstyle ({y\over z})^{i-1}}\left(D_{y}^{[i]}\right)^{-1}\! 
R_{B_+}^{(n+i+1)}\; D_{x z/y}^{[i]}, & 
\\[5pt]
 R_{B_-}^{(n+i+1)} &\mapsto&  {\textstyle ({z\over y})^{i-1}}\left(D_{x z/ y}^{[i]}\right)^{-1}\! 
R_{B_-}^{(n+i+1)}\; D_{y}^{[i]},
 &i=1,\dots , n-1;
\\[12pt]
R_{C}^{(2n+i)}&\mapsto& \left(D_{z}^{[i-1]}\right)^{-1}\! R_{C}^{(2n+i)}\, D_{z}^{[i-1]}, & 
\\[5pt]
R_{B_0}^{(2n+i)} &\mapsto&  \left(D_{y}^{[n-i]}\right)^{-1}\! R_{B_0}^{(2n+i)}\; D_{y}^{[n-i]}, &
\\[5pt]
R_{B_+}^{(2n+i)} &\mapsto & {\textstyle {x^{i-1}y^{n-1-i}\over z^{n-1}}}  \left(D_{y}^{[n-i]}\right)^{-1}\! 
R_{B_+}^{(2n+i)}\; D_{x z/y}^{[n-i]}, &
\\[5pt]
R_{B_-}^{(2n+i)} &\mapsto& {\textstyle {z^{n-1}\over x^{i-1}y^{n-1-i}}}  \left(D_{x z/ y}^{[n-i]}\right)^{-1}\! 
R_{B_-}^{(2n+i)}\; D_{y}^{[n-i]},&i=1,\dots , n-1;
\\[12pt]
R^{(3n+i-1)}&\mapsto& \left(D_{z}^{[n-i]}\right)^{-1}\! R^{(3n+i-1)}\, D_{z}^{[n-i]},& i=1,\dots , n-1.
\end{array}
\ee
The correspondence between the twist parameters $y$, $z$ and the parameters $x'$, $x''$ that we introdused in the proof of proposition \ref{prop3} is as follows: $x'\leftrightarrow z/y$,~ $x''\leftrightarrow 1/y$. The twist parameter $x$ first appeared in proposition \ref{prop1} (see eq.(\ref{pre-eq})).  In the formulas for the R-matrix in proposition \ref{prop3} it was set equal to 1.

\subsection{R-matrix for $q$ a primitive $n$-th root of unity, $\mbox{dim}\, V=2n$}

A family of the diagonal twist transformations of  R-matrix (\ref{C1})-(\ref{C9}), parameterized by three independent parameterts $x, y, z\in {\Bbb C}\!\setminus\! \{0\}$ is given on its blocks as follows
\be
\lb{twist-C}
\begin{array}{lclr}
	R^{(i+1)}&\mapsto& \left(D_x^{[i]}\right)^{-1}\! R^{(i+1)}\, D_x^{[i]},& i=1,\dots , n;
	\\[12pt]
	R_{C}^{(n+i+1)}&\mapsto&  \left(D_{x}^{[n-i]}\right)^{-1}\!  R_{C}^{(n+i+1)}\, D_{x}^{[n-i]}, & 
	\\[5pt]
	R_{B_0}^{(n+i+1)} & \mapsto&   \left(D_{y}^{[i]}\right)^{-1}\! R_{B_0}^{(n+i+1)}\; D_{y}^{[i]}, &
	\\[5pt]
	R_{B_+}^{(n+i+1)} &\mapsto &  {\textstyle ({y\over z})^{i-1}}\left(D_{y}^{[i]}\right)^{-1}\! 
	R_{B_+}^{(n+i+1)}\; D_{x z/y}^{[i]}, & 
	\\[5pt]
	R_{B_-}^{(n+i+1)} &\mapsto&  {\textstyle ({z\over y})^{i-1}}\left(D_{x z/ y}^{[i]}\right)^{-1}\! 
	R_{B_-}^{(n+i+1)}\; D_{y}^{[i]},
	&i=1,\dots , n;
	\\[12pt]
	R_{C}^{(2n+i+1)}&\mapsto& \left(D_{z}^{[i]}\right)^{-1}\! R_{C}^{(2n+i+1)}\, D_{z}^{[i]}, & 
	\\[5pt]
	R_{B_0}^{(2n+i+1)} &\mapsto&  \left(D_{y}^{[n-i]}\right)^{-1}\! R_{B_0}^{(2n+i+1)}\; D_{y}^{[n-i]}, &
	\\[5pt]
	R_{B_+}^{(2n+i+1)} &\mapsto & {\textstyle {x^{i}y^{n-1-i}\over z^{n-1}}}  \left(D_{y}^{[n-i]}\right)^{-1}\! 
	R_{B_+}^{(2n+i+1)}\; D_{x z/y}^{[n-i]}, &
	\\[5pt]
	R_{B_-}^{(2n+i+1)} &\mapsto& {\textstyle {z^{n-1}\over x^{i}y^{n-1-i}}}  \left(D_{x z/ y}^{[n-i]}\right)^{-1}\! 
	R_{B_-}^{(2n+i+1)}\; D_{y}^{[n-i]},&i=1,\dots , n-1;
	\\[12pt]
	R^{(3n+i+1)}&\mapsto& \left(D_{z}^{[n-i]}\right)^{-1}\! R^{(3n+i+1)}\, D_{z}^{[n-i]},& i=0,\dots , n-1.
\end{array}
\ee

\end{document}